\documentclass{amsart}
\usepackage{amssymb}

\newtheorem{theorem}{Theorem}[section]
\newtheorem{lemma}[theorem]{Lemma}

\theoremstyle{definition}

\theoremstyle{remark}

\numberwithin{equation}{section}

\begin{document}

\title[On Davenport's constant]{Davenport's constant for groups with
  large exponent}


\author{}
\address{}
\curraddr{}
\email{}
\thanks{}

\author[G. Bhowmik]{Gautami Bhowmik}
\address{Universit\'e de Lille 1\\
Laboratoire Paul Painlev\'e UMR CNRS 8524\\ 
59655 Villeneuve d'Ascq Cedex\\
France}
\curraddr{}
\email{bhowmik@math.univ-lille1.fr}
\thanks{}

\author[J.-C. Schlage-Puchta]{Jan-Christoph Schlage-Puchta}
\address{Universiteit Gent\\ Krijgslaan 281\\ Gebouw S22\\ 9000
  Gent\\ Belgium}
\email{jcp@math.uni-freiburg.de}

\subjclass[2000]{Primary 11B13, 11B70}

\date{}

\begin{abstract}
Let $G$ be a finite abelian group. We show that its Davenport constant
$D(G)$ satisfies $D(G)\leq
\exp(G)+\frac{|G|}{\exp(G)}-1$, provided that $\exp(G)\geq\sqrt{|G|}$,
and $D(G)\leq 2\sqrt{|G|}-1$, if $\exp(G)<\sqrt{|G|}$. This proves a
conjecture by Balasubramanian and the first named author.
\end{abstract}

\maketitle


\bibliographystyle{amsplain}
\newcommand{\Z}{\mathbb{Z}}
\newcommand{\s}{\mathfrak{s}}
\newcommand{\F}{\mathbb{F}}
\section{Introduction and results}
For an abelian group $G$ we denote by $D(G)$ the least integer $k$, such
that every sequence $g_1, \ldots, g_k$ of elements in $G$ contains a
subsequence $g_{i_1}, \ldots, g_{i_\ell}$ with
$g_{i_1}+\dots+g_{i_\ell}=0$.

Write $G=\Z_{n_1}\oplus\dots\oplus\Z_{n_r}$ with $n_1|\dots|n_r$,
where we write $\Z_n$ for $\Z/n\Z$. Put $M(G)=\sum n_i -r+1$.
In several cases, including 2-generated groups and $p$-groups, the
value of $D(G)$ matches with the obvious lower bound $M(G)$, however, in
general this is not true. In fact there are infinitely many groups of
rank 4 or more where $D(G)$ is greater than $M(G)$
 see, for example, \cite{barcelona}. As far as upper bounds are
 concerned we have only rather crude ones.
 One such example, which 
is appealing for its simple structure, is the estimate $D(G)\leq\exp(G)
\big(1+\log\frac{|G|}{\exp(G)}\big)$, due to van Emde Boas and
Kruyswijk\cite{EBK}. This bound, for the case when
$\frac{|G|}{\exp(G)}$ is small, was improved by Bhowmik and
Balasubramanian \cite{BaB} who proved that
$D(G)\leq\frac{|G|}{k}+k-1$, where $k$ is an
integer
$\leq\min(\frac{|G|}{\exp(G)}, 7)$, and conjectured that one may
replace the constant $7$ by $\sqrt{|G|}$. Here we prove this
conjecture. It turns out that the
hypothesis that $k$ be integral creates some technical difficulties,
therefore we prove the following, slightly sharper result.
\begin{theorem}
\label{thm:main}
For an abelian group $G$ with $\exp(G)\geq\sqrt{|G|}$ we have
$D(G)\leq\exp(G)+\frac{|G|}{\exp(G)}-1$, while for
$\exp(G)<\sqrt{|G|}$ we have $D(G)\leq 2\sqrt{|G|}-1$.
\end{theorem}
We notice that the first upper bound is actually reached for groups of rank 2
where $D(G)=\exp(G)+\frac{|G|}{\exp(G)}-1$. An application of our
bound to random groups and $(\Z_n^*, \cdot)$ will be the topic of a
forthcoming paper.

Let  $\s_{\leq n}(G)$ be the least integer
$k$, such that every sequence of length $k$ contains a subsequence
of length $\leq n$ adding up to 0  and let $\s_{=n}(G)$ be the least
integer $k$ such that any sequence of length $k$ in $G$
contains a zero-sum of sequence of length exactly equal to $n$.
In the special case  where $n=\exp(G)$ we use the more standard notation of
$\eta(G)$  and  $\s(G)$ respectively. We need the following bounds on
$\eta$ and $\s$.
\begin{theorem}
\label{thm:eta}
\begin{enumerate}
\item We have $\s(\Z_3^3)=19$, $\s(\Z_3^4)=41$, $\s(\Z_3^5)=91$, and
  $\s(\Z_3^6)= 225$. 
\item We have $\s(\Z_5^3)=37$, $\s(\Z_5^4)\leq 157$, $\s(\Z_5^5)\leq 690$, and
  $\s(\Z_5^6)\leq 3091$.
\item If $p\geq 7$ is prime and $d\geq 3$, then $\eta(\Z_p^d)\leq
  \frac{p^d-p}{p^2-p}(3p-7)+4$. 
\end{enumerate}
\end{theorem}
The  above results for $\Z_3$ are due to Bose\cite{Bose},
Pelegrino\cite{Pel}, Edel, 
Ferret, Landjev and Storme\cite{EFLS}, and Potechin\cite{Pot},
respectively. The value of $\s(\Z_5^3)$ was determined by Gao, Hou,
Schmid and Tangadurai\cite{GHST}, the bounds for higher rank will be proven in
section~\ref{sec:density} using the density increment method. The last
statement will be proven by combinatorial means in
section~\ref{sec:induct}. 

We further need some information on the existence of zero-sums not
much larger then $\exp(G)$.
\begin{theorem}
\label{thm:short}
Let $p$ be a prime, $d\geq 3$ an integer. Then a sequence of length
$(6p-4)p^{d-3}+1$ in $\Z_p^d$ contains a zero-sum of length
$\leq\frac{3p-1}{2}$. If $d\geq 4$, then a sequence of length
$(6p-4)p^{d-4}+1$ in $\Z_p^d$ contains a zero-sum of length $\leq
2p$.
\end{theorem}

The proof of  Theorem~\ref{thm:main} uses the {\it inductive method}. To deal
with the inductive step we require the following.

\begin{theorem}
\label{thm:step}
Let $p$ be a prime, $d\geq 2$ an integer. Then there exist integers
$k, M$, such that $M\geq\eta(\Z_p^d)$, every sequence of length $M$
contains at least $k$ disjoint zero-sums, and $M\leq p^{d-1}+pk$.
\end{theorem}

Note that the statement is trivial if $\eta(\Z_p^d)\leq p^{d-1}$. However, this
bound is false for $p=2$ and all $d$, as well as for the pairs $(3,3),
(3,4), (3,5)$ and $(5,3)$. We believe that this is the complete list
of exceptions. From the Alon-Dubiner-theorem and Roth-type
estimates one can already deduce that the above bound for $\eta$ holds for all but finitely
many pairs. However, dealing with the exceptional pairs by direct
computation is way beyond current computational means. 

\section{Systems of disjoint zero-sums} 
Let $D_k(G)$ be the least integer $t$ such that every sequence of
length $t$ in $G$ contains $k$ disjoint zero-sum sequences. The most
direct way to prove the existence of many disjoint zero-sums is by
proving the existence of rather short zero-sums, therefore we are
interested in zero sums of length not much beyond $p$.

\begin{lemma}
\label{Lem:short}
Every sequence of length $6p-3$ in $\Z_p^3$ contains a zero-sum of
length $\leq\frac{3p-1}{2}$, every sequence of length $6p-3$ in
$\Z_p^4$ contains a zero-sum of length $\leq 2p$, and every sequence
of length $(d+1)p-d$ in $\Z_p^d$ contains a zero sum of length $\leq (d-1)p$.
\end{lemma}
\begin{proof}
We claim that a sequence of length $6p-3$ in $\Z_p^3$ contains a zero sum of
length $p$ or $3p$. To see this we adapt Reiher's proof of Kemnitz'
conjecture \cite{ReiherKemnitz}. For a sequence $S$ denote by
$N^\ell(S)$ the
number of zeo-sum subsequences of $S$ of length $\ell$. Let $S$ be a
seauence of length $6p-3$ without a zero sum of length $p$ or $3p$,
$T$ a subsequence of length $4p-3$, and $U$ a subsequence of length
$5p-3$. Then the Chevalley-Warning theorem gives the following
equations.
\[
\begin{array}{lcl}
1+N^p(T)+N^{2p}(T)+N^{3p}(T) & \equiv & 0\pmod{p}\\
1+N^p(U)+N^{2p}(U)+N^{3p}(U)+N^{4p}(U) & \equiv & 0\pmod{p}\\
1+N^p(S)+N^{2p}(S)+N^{3p}(S)+N^{4p}(S)+N^{5p}(S) & \equiv & 0\pmod{p}\\
\end{array}
\]
By assumption $S$, and a forteriori $U$ and $T$ do not contain zero
sums of length $p$ or $3p$, thus all occurrences of $N^p$ and $N^{3p}$
vanish. If $N^{5p}(S)\neq 0$, and $Z$ is a zero sum in $S$, then
choosing for $T$ a subsequence of $Z$ of length $4p-3$ we find from
the first equation that $T$ contains a zero sum $Y$ of length
$2p$. But then $Z\setminus Y$ is a zero sum of length $3p$, a
contradiction. We now add up the first equation over all subsequences
$T$ of length $4p-3$, and the second over all subsequences of length
$5p-3$, and obtain a system of three euations in the two variables
$N^{2p}(S)$, $N^{4p}(S)$, which is unsolvable. 

Now let $S$ be a sequence of length $6p-3$, and let $Z$ be a zero sum
of length $p$ or $3p$. If $|Z|=p$, then we found a zeo sum of length
$\leq\frac{3(p-1)}{2}$. Otherwise $Z$ contains a zero sum $Y$, and
then either $Y$ or $Z\setminus Y$ is the desired zero-sum of length $\leq
\frac{3(p-1)}{2}$.

The second claim follows similarly starting from the fact that every
sequence of length $6p-3$ in $\Z_p^4$ contains a zero-sum subsequence
of length $p,2p$ or $4p$, while the last one follows from the fact
proven by Gao and Geroldinger\cite[Theorem~6.7]{GG}, that a
sequence of length $(d+1)p-d$ contains a zero sum of length divisible
by $p$.
\end{proof}

The next result is used to lift results for special groups $\Z_p^d$ to
groups of arbitrary rank. The argument is rather wasteful, still
the resulting bounds are surprisingly useful.

\begin{lemma}
\label{Lem:lift}
If $a\leq d$, then $\mathfrak{s}_{\leq k}(\Z_p^d)\leq\frac{p^d-1}{p^a-1}
(\mathfrak{s}_{\leq k}(\Z_p^a)-1)+1$
\end{lemma}
\begin{proof}
Let $A$ be a sequence of length $\frac{p^d-1}{p^a-1}
(\mathfrak{s}_{\leq k}(\Z_p^a)-1)+1$ in $\Z_p^d$. If $A$ contains 0,
then we found a short zero sum. Otherwise let $U$ be a subgroup of
$\Z_p^d$ with $U\cong\Z_p^a$ chosen at random. The expected number of
elements of $A$, which are in $U$ is sightly bigger than
$\mathfrak{s}_{\leq k}(\Z_p^a)-1$, hence there exists a subgroup which
contains at least $\mathfrak{s}_{\leq k}(\Z_p^a)$ elements of the sequence.
Restricting our attention to this subgroup we obtain the desired zero sum. 
\end{proof}

\begin{lemma}
\label{Lem:shortsums}
We have
\[
D_k(\Z_p^3) \leq \max\Big(5p-2, \frac{3(p-1)}{2}+2p+5\Big),
\]
and, for $d\geq 4$,
\[
D_k(\Z_p^d) \leq \max\big((6p-4)p^{d-3}+1, \frac{3(p-1)}{2}k + 1 +
(6p-4)p^{d-3}\big(\frac{1}{4}+\frac{3}{2p}-\frac{3}{4p^2}-\frac{1}{dp}\big)\Big)
\]
\end{lemma}
\begin{proof}
We only give the proof for the second inequality, the first one being
significantly easier.

Let $S$ be a sequence of length at least $(6p-4)p^{d-3}+1$. Then we
can find a zero sum of length 
$\leq\frac{3(p-1)}{2}$. We continue doing so until there are less
zero-sums left. Then we remove zero sums of length $\leq 2p$, until
there are less than $(6p-4)p^{d-4}+1$ points left. Among the remaining
points we still find zero sums of length at most $D(\Z_p^d)=d(p-1)+1$,
hence, in total we obtain a system of at least
\[
\frac{|S|-(6p-4)p^{d-3}-1}{3(p-1)/2} +
\frac{(6p-4)p^{d-3}-(6p-4)p^{d-4}}{2p} + \frac{(6p-4)p^{d-4}}{d(p-1)+1}
\]
disjoint zero sums. Hence, 
\begin{multline*}
D_k(\Z_p^d) \leq (6p-4)p^{d-3}+1+\\
\max\Big(0, \frac{3(p-1)}{2}
\big(k-\frac{(6p-4)p^{d-3}-(6p-4)p^{d-4}}{2p} + \frac{(6p-4)p^{d-4}}{d(p-1)+1}
\big)\Big),
\end{multline*}
and our claim follows.
\end{proof}

The reader should compare our result with a similar bound given by
Freeze and Schmid\cite[Proposition~3.5]{FS}. In our result the
coefficient of $k$ is smaller, while the constant term is much bigger.
The following result is an interpolation between these results.

\begin{lemma}
\label{Lem:intermediate}
Let $N, d\geq 3$ be integers, $p$ a prime number, and define $a$ to be the
largest integer such that $N>(a+1)p^{d-a+1}$. If $a\geq 2$, then
$D_k(\Z_p^d)\leq N$, where
\[
k=\frac{N}{(a-1)p}-\sum_{\nu=a}^{d-1}\frac{\nu+1}{\nu(\nu-1)}p^{d-a}
-1 \geq \frac{N}{(a-1)p}\big(1-\frac{1}{a(1-p^{-1})}\big)
\]
\end{lemma}
\begin{proof}
Let $S$ be a sequence of length $N$ in $\Z_p^d$. We have to show that
$S$ contains a systm of $k$ disjoint zero sums. 
Since $N>(a+1)p^{d-a+1}$, $S$ contains a zero sum of length $\leq
(a-1)p$. We remove zero sums of this length, until the remaining
sequence has length $<(a+1)p^{d-a+1}$. From this point onward we
remove zero sums of length $\leq ap$, until the remainder has length
$<(a+2)p^{d-a+2}$, and so on. In this way we obtain a disjoint system
consisting of
\[
\frac{N-(a+1)p^{d-a+1}}{(a-1)p} +
\frac{(a+1)p^{d-a+1}-(a+2)p^{d-a+2}}{ap} + \dots + \frac{dp^2-(d+1)p}{ap}+1
\]
zero sums. This sum almost telescopes, yielding the first expression
for $k$. For the inequality note that the sequence
$\frac{\nu+1}{\nu(\nu-1)}$ is decreasing, hence the summands in the
series are decreasing faster than the geometrical series $\sum
p^{-\nu}$, and we conclude that the whole sum is bounded by the first
summand multiplied by $(1-p^{-1})^{-1}$. Our claim now follows.
\end{proof}
The following result is a special case of a result of
Lindstr\"om\cite{Lind} (see also \cite[Theorem~7.2, Lemma~7.4]{FS}).
\begin{lemma}\label{Lem:mediumsums}
Every sequence of length $2^{d-1}+1$ in $\Z_2^d$ contains a zero-sum
of length $\leq 3$, and this bound is best possible. Every sequence of
length $2^{(d+1)/2}+1$ in $\Z_2^d$ contains a zero-sum of length $\leq
4$.
\end{lemma}

\section{Proof of Theorem~\ref{thm:main}}
\label{sec:main}

In this section we show that Theorem~\ref{thm:step} implies
Theorem~\ref{thm:main}. 

\begin{lemma}
Let $G$ be an abelian group of rank $r\geq 3$. Assume that
Theorem~\ref{thm:main} holds true for all proper subgroups of
$G$. Then it holds true for $G$ itself.
\end{lemma}
\begin{proof}
Let $p$ be a prime divisor of
$|G|$. Choose an elementary abelian 
subgroup $U\cong \Z_p^d$ of $G$, such that $d\geq 3$,
$\exp(G)=p\exp(G/U)$, and $|U|$ is minimal under these
assumptions. Put $H=G/U$. Let $A$ be a sequence consisting of
$\exp(G)+\frac{|G|}{\exp(G)}-1$ or $2\lfloor\sqrt{|G|}\rfloor-1$
elements, depending on whether $\exp(G)>\sqrt{|G|}$ or not. Denote by
$\overline{A}$ the image of $A$ in $H$. Then we obtain a zero-sum, by
choosing a large system of disjoint zero-sums in $\Z_p^d$, and then
choosing a zero-sum among the elements in $H$ defined by these sums,
provided that
\[
D(H)\leq\frac{|A|-M}{p}+k,
\]
where $M\geq\eta(\Z_p^d)$ and $k=k(p, d, M)$ is defined as in
Theorem~\ref{thm:step}. The left hand side can be estimated using the 
inductive hypothesis. We have $\exp(H)=\frac{\exp(G)}{p}$,
$|H|=\frac{|G|}{p^d}$. Assume first that $\exp(G)\geq\sqrt{|G|}$ and
$\exp(H)\geq\sqrt{|H|}$. Then our claim follows, provided that
\[
\frac{\exp(G)}{p}+\frac{|G|}{\exp(G)p^d}-1\leq \frac{|A|-M}{p}+k,
\]
inserting the choice of $A$ and rearranging terms this becomes
\[
\exp(G)+\frac{|G|}{\exp(G)p^{d-1}}-p\leq
\exp(G)+\frac{|G|}{\exp(G)}-1-M+pk. 
\]
The quotient of $G$ by its largest cyclic subgroup contains at least
$\Z_p^{d-1}$, hence, $\frac{|G|}{\exp(G)}\geq p^{d-1}$. Clearly, by
replacing $\frac{|G|}{\exp(|G|}$ with a lower bound we lose something,
hence, it suffices to establish the relation
\[
1-p\leq p^{d-1}-1-M+pk.
\]
However, this relation is implied by Theorem~\ref{thm:step}.

Next suppose that $\exp(G)\geq\sqrt{|G|}$ and $\exp(H)<\sqrt{|H|}$.
Then
\[
\sqrt{|G|/p^d} = \sqrt{|H|} > \exp(H)=\exp(G)/p \geq\sqrt{|G|/p^2},
\]
thus $d<2$, but this case was excluded from the outset.

If $\exp(G)<\sqrt{|G|}$ and $\exp(H)<\sqrt{|H|}$, the same argument as
in the first case yields $D(G)\leq 2\sqrt{G}-1$, provided that
\[
2p\sqrt{|H|}-p\leq2\sqrt{|G|}-1-M+pN.
\]
Since $|H|=\frac{|G|}{p^d}$ and $M-pN\leq p^{d-1}$ this becomes
\[
(2-2p^{-(d-2)/2})\sqrt{|G|} \geq p^{d-1}-p+1.
\]
As $\exp(H)<\sqrt{H}$ we have that $H$ is of rank at least 3, which by
our assumption on the size of $H$ implies that $|G|\geq p^{2d}$. This
implies 
\[
(2-2p^{-(d-2)/2})\sqrt{|G|}\geq (2-2p^{-(d-2)/2})p^d > \frac{1}{2}p^d > p^{d-1}-p+1,
\]
and our claim is proven.

If $\exp(G)<\sqrt{|G|}$ and $\exp(H)\geq\sqrt{H}$, the theorem follows
provided that
\[
\big(\exp(H)+\frac{|H|}{\exp(H)}-1\big)p \leq 2\sqrt{|G|}-1 - M +kp,
\]
that is
\[
\exp(G)+\frac{|G|}{p^{d-2}\exp(G)}-p \leq 2\sqrt{|G|}-1 - p^{d-1}.
\]
The bounds for $\exp(G)$ and $\exp(H)$ imply $\sqrt{|G|}
p^{d/2-1}\leq\exp(G) <\sqrt{|G|}$, and in this range the left hand
side is increasing as a function of $\exp(G)$, hence, this inequality
is certainly true if 
\[
\sqrt{|G|} \geq 1+p^{d-1} + \sqrt{|G|} p^{2-d} - p,
\]
which follows from $\sqrt{|G|}\geq p^d$. If this is not the case, then
$|H|<p^d$, and by the choice of $p$ we have that $H$ has rank at most
2, that is, $H=\Z_{n_1}\oplus\Z_{n_2}$ and
$G=\Z_p^{d-2}\oplus\Z_{pn_1}\oplus\Z_{pn_2}$, say. Then
$D(H)=n_1+n_2-1$, thus it suffices to prove
\[
D_{n_1+n_2-1}(\Z_p^d)\leq 2\sqrt{p^dn_1n_2}-1.
\]
Denote the right hand side by $N$. Then Lemma~\ref{Lem:intermediate}
shows that our claim holds true, provided that
\[
n_1+n_2-1\leq \frac{N}{(a-1)p}\big(1-\frac{1}{a(1-p^{-1})}\big).
\]
Using the trivial bound $n_1+n_2-1\leq n_1n_2$ we find that this
inequality follows from
\[
\frac{ap^{d-a}}{(a-1)p}\big(1-\frac{1}{a(1-p^{-1})}\big) \geq
\frac{a+p^{-(d-a)}}{4a}\big(ap^{-a} + p^{-d}\big),
\]
and by direct inspection we see that our claim follows for all $a\geq
2$, with exception only the case $(p,a)=(2,2)$. In this case our claim
follows from Lemma~\ref{Lem:mediumsums}, provided that $d>3$. Finally,
if $p=2$ and $d=3$, then $D(G)=M(G)$ was shown by van Emde Boas\cite{EB} under
the assumption that Lemma~\ref{Lem:PropC} holds true for all prime
divisors of $|H|$, which we today know to hold for all primes. Hence the proof
is complete.
\end{proof}

We know that $D(\Z_{n_1}\oplus\Z_{n_2}) = n_1+n_2-1$, hence
Theorem~\ref{thm:main}  holds true for all groups of rank 
$\leq 2$. Hence Theorem~\ref{thm:main} follows by induction over the
group order.

\section{Proof of Theorem~\ref{thm:step}: The case $p\leq 7$}
\label{sec:density}

\subsection{The primes 2 and 3}
 


To prove Theorem~\ref{thm:step} for $p=2$, we want to show that in a
set of $2^d$ points we can find a system 
consisting of many disjoint zero-sums. We first remove one zero-sum of
length $\leq 2$, then zero-sums of
length $\leq 3$, until this is not possible anymore, and then we
switch to zero-sums of length 4. Finally we remove zero-sums of length
$\leq d+1$, which is possible in view of $D(\Z_2^d)=d+1$. In this
way
we obtain at least 
\begin{multline*}
\frac{2^d-2}{3} + \frac{2^{d-1}+2-2^{(d+1)/2}-1}{4} +
\frac{2^{(d+1)/2}-d-2}{d+1}+1 =\\ 
\frac{2^d}{4} + \frac{2^d+2}{24} - 2^{(d-3)/2} + \frac{2^{(d+1)/2}-1}{d+1}
\end{multline*}
zero-sums. Disregarding the last fraction we see that this quantity
is $\geq 2^{d-2}$, provided that $d\geq 7$. For $3\leq d\leq 6$ we
obtain our claim by explicitly computing this bound. 

Next we consider $p=3$. For $d\geq 6$ we have
\[
\eta(\Z_3^d)\leq \s(\Z_3^d) \leq 3^{d-6} \s(\Z_3^6) < 3^{d-1},
\]
hence, Theorem~\ref{thm:step} holds true with $N=0$, $M=3^{d-1}$. For $d=5$
it follows from Lemma~\ref{Lem:short} that a sequence of length 
$\eta(\Z_3^5)-3$ contains a system of 
$N=\lceil\frac{\eta(\Z_3^5)-2d-6}{3d-3}\rceil$ disjoint zero-sums,
hence, our claim follows provided that
\[
\eta(\Z_3^5)\leq 3\lceil\frac{\eta(\Z_3^5)-16}{12}\rceil + 3^4,
\]
that is, $89\leq 21+81$. In the same way we see that for $d=4$ a
sequence of length 39 in $\Z_3^4$ contains a system of 4 disjoint
zero-sums, thus our claim follows from $39\leq 12+27$. Finally it is
shown in \cite[Proposition~1]{Dav333d}, that a sequence of length 15
in $\Z_3^3$ contains a system of 3 disjoint zero-sums. Together with
$\eta(\Z_3^3)=17$ our claim follows in this case as well.

\subsection{The prime 5}

We begin by proving the second statement of
Theorem~\ref{thm:eta}. We do so by using a density increment argument
together with 
explicit calculations. Define the Fourier bias $\|A\|_u$ of a sequence
$A$ over $\F_p^d$ as
\[
\|A\|_u := \frac{1}{|A|}\max_{\xi\in\F_p^d\setminus\{0\}}
\sum_{\alpha\in A} e(\langle\xi, \alpha\rangle).
\]
Then we have the following.
\begin{lemma}
Let $p\geq3$ be a prime number, $A$ be a sequence over $\F_p^d$. Then
$A$ contains a zero-sum of length $p$, provided that
\[
\frac{|A|^{p-1}}{p^{(p-1)d}} >
\|A\|_u^{p-3}\left(\|A\|_u+\frac{p-1}{2p^{d-1}}\right)
 +
\binom{p}{2}\frac{|A|^{p-2}}{p^{(p-1)d}}
\]
\end{lemma}
\begin{proof}
Let $N$ be the number of solutions of the equation $a_1+\dots+a_p = 0$
with $a_i\in A$. From \cite[Lemma~4.13]{TaoVu} we have
\[
N \geq \frac{|A|^p}{p^d} - \|A\|_u^{p-2}|A|p^{(p-2)d}.
\]
A solution $a_1+\dots+a_p = 0$ corresponds to a zero-sum of $A$, if $a_1,
\ldots, a_p$ are distinct elements in $A$. Using M\"obius inversion
over the lattice of set partitions one could compute the over-count
exactly, however, it turns out that the resulting terms are of
negligible order, which is why we bound the error rather crudely. The
number of solutions $M$ in which not all elements are different is at most
$\binom{p}{2}$ times the number of solutions of the equation
$2a_1+a_2+\dots+a_{p-1}=0$. Since multiplication by 2 is a linear map
in $\F_p^d$ we have that $\|2A\|_u=\|A\|$, using
\cite[Lemma~4.13]{TaoVu} again we obtain
\[
M\leq \frac{|A|^{p-1}}{p^d} + \|A\|_u^{p-3}|A|p^{(p-3)d}.
\]
Hence the number of zero-sums is at least
\[
N-M\geq \frac{|A|^p}{p^d} - \|A\|_u^{p-2}|A|p^{(p-2)d} -
\frac{|A|^{p-1}}{p^d} - \|A\|_u^{p-3}|A|p^{(p-3)d},
\]
and our claim follows.
\end{proof}

We now use this lemma recursively to obtain bounds for $\s(\Z_5^d)$,
starting from $\s(\Z_5^3)=37$.

Consider a 3-dimensional subgroup $U$, and let $\xi\in\Z_5^4$ be a
vector such that $v\bot U$. Let $n_1, \ldots, n_5$ be the number of
elements of $A$ in each of the 5 cosets of $U$, $\zeta$ be a fifth
root of unity. If $\max(n_i)\geq 37$, we have a zero-sum of length $p$
in one of the hyperplanes. Hence
\[
\|A\|_u\leq \frac{1}{|A|}\underset{0\leq n_i\leq 36}{\max_{n_1+\dots+n_5=|A|}}
|n_1+n_2 \zeta + \dots + n_5\zeta^4|.
\]
Since $1+\zeta+\dots+\zeta^4=0$, we have
\[
n_1+n_2 \zeta + \dots + n_5\zeta^4 = (36-n_1)+(36-n_2) \zeta + \dots +
(36-n_5)\zeta^4,
\]
that is, 
\[
\underset{0\leq n_i\leq 36}{\max_{n_1+\dots+n_5=|A|}}
|n_1+n_2 \zeta + \dots + n_5\zeta^4| = \underset{0\leq n_i\leq 36}{\max_{n_1+\dots+n_5=180-|A|}}
|n_1+n_2 \zeta + \dots + n_5\zeta^4|.
\]
For $|A|\geq 144$ the right hand side equals $180-|A|$, and we obtain
a zero-sum, provided that
\[
\left(\frac{|A|}{625}\right)^4 > \left(\frac{180-|A|}{|A|}\right)^2
\left(\frac{180-|A|}{|A|} + \frac{2}{125}\right) + \frac{2}{125}\left(\frac{|A|}{625}\right)^3.
\]
One easily finds that this is the case for $|A|=157$, and we deduce
$\s(\Z_5^4)\leq 157$. The same argument yields for $d=5$ the
inequality
\[
\left(\frac{|A|}{3125}\right)^4 > \left(\frac{780-|A|}{|A|}\right)^2
\left(\frac{780-|A|}{|A|} + \frac{2}{625}\right) + \frac{2}{625}\left(\frac{|A|}{3125}\right)^3,
\]
which is satisfied for $|A|\geq 690$, that is, we obtain
$\s(\Z_5^5)\leq 690$. Finally for $\Z_5^6$ we obtain
\[
\left(\frac{|A|}{15625}\right)^4 > \left(\frac{3445-|A|}{|A|}\right)^2
\left(\frac{3445-|A|}{|A|} + \frac{2}{3125}\right) + \frac{2}{3125}\left(\frac{|A|}{3125}\right)^3,
\]
which is satisfied for $|A|\geq 3091$, thus the last inequality
follows as well.

Hence, Theorem~\ref{thm:eta}(2) is proven.

We have $\eta(\Z_5^3)=33$, and among $33$ elements
we can find one zero-sum of length $\leq 5$, one of length $\leq 10$,
and one more among the remaining $18\geq D(\Z_5^3)=13$ points. Hence we can
take $M=33, N=3$, and Theorem~\ref{thm:step} follows. Moreover we
have $\eta(\Z_5^4)\leq \s(\Z_5^4)-4\leq 153$, and among 153 elements
we can find one zero-sum of 
length $\leq 5$, 13 zero-sums of length $\leq 10$, and one more
zero-sum, that is, we can take $N=15$, and Theorem~\ref{thm:step}
follows for $d=4$ as well.

For $d=5$ we have $\eta(\Z_5^5)\leq\s(\Z_5^5)-4\leq 686$, and
among 686 points in $\Z_5$ we find 24 disjoint zero-sums of length
$\leq 20$, thus taking $M=686$, $N=24$, our claim follows since $M\leq
625+120$. For $d\geq 6$ we have
\[
\s(\Z_5^d)\leq 5^{d-6}\s(\Z_5^6)\leq 3091\cdot 5^{d-6} < 5^{d-1},
\]
and our claim becomes trivial.

\section{Proof of Theorem~\ref{thm:step}: The case $p\geq 7$}
\label{sec:induct}

We begin by proving the last statement of Theorem~\ref{thm:eta}.

\begin{lemma}
\label{Lem:PropC}
Let $A$ be a sequence of length $3p-3$ in $\Z_p^2$ without a zero-sum
of length $\leq p$. Then $A=\{a^{p-1}, b^{p-1}, c^{p-1}\}$ for
suitable elements $a, b, c\in\Z_p^2$.
\end{lemma}
\begin{proof}
A prime $p$ is said to satisfy {\em property B} if in every maximal
zero-sum free subset of $\Z_p^2$ some element occurs with multiplicity
at least $p-2$. Gao and Geroldinger\cite{GGIntegers} have shown that
the condition of the above lemma  holds true 
if $p$ has property B, and Reiher\cite{ReiherB} has shown that every
prime has property B.
\end{proof}
For $p=7$ we need a little more specific information.
\begin{lemma}
Let $A$ be a sequence of length 15 over $\Z_7^2$, which does not
contain a zero-sum of length $\leq 7$. Then there exist a cyclic
subgroup which contains 3 elements of $A$.
\end{lemma}
\begin{proof}
The proof can be done either by a mindless computer calculation or by
a slightly more sophisticated human readable argument, however, as the
latter also boils down to a sequence of case distinction we shall be a
little brief. Let $A$ be a counterexample, that is, a
zero-sum free sequence of length 15, such that every cyclic subgroup
contains at most 15 points. We shall deduce properties of $A$ in a
bootstrap manner.

{\it Without loss we may assume that $A$ contains no two elements $x,
  y$ with $y=2x$.} Suppose that $A$ contains two such elements. Then
replacing $y$ by $x$ gives a new sequence $A'$, such that for an element
in $\Z_p^2$ the shortest representation as a subsum of $A'$ is at
least as long as the shortest representation as a subsum of $A$. In
particular, $A'$ contains no short zero sum.

{\it There is at most one subgroup which contains two different
  elements.} Without loss we may assume that $(1,0), (3,0), (0,1),
(0,3)$ are in $A$. The subgroup generated by $(1,1)$ can contain
either $(5,5)$ with multiplicity 2, or one of $(1,1),(2,2), (5,5)$
with multiplicity 1. If $(5,5)$ occurs twice, the remaining elements
of the sequence must be among $\{(2,3), (2,4), (3,2), (3,5), (4,2),
(4,5), (5,3), (5,4)\}$, which can easily be ruled out.  If
$(5,5)$ does not occur twice, then all subgroups different from
$\langle(1,0)\rangle, \langle(0,1)\rangle, \langle(1,1)\rangle$
contain one element with multiplicity 2. The only possible elements in
$\langle(1, -1)\rangle$ are $(1,6), (6, 1)$, and by symmetry we may
assume that $(6,1)$ occurs twice. Now $\langle(1,2)\rangle$ must
contain $(6,5)$, and we conclude that the remaining points are $(2,6)$,
$(3,1)$, $(5,4)$, and we obtain the zero-sum
$(5,4)+(6,1)+(3,1)+(0,1)$.

{\it There exist 3 different elements $x, y, z$, each of multiplicity
  2 in $A$, such that $x+y\in\langle z\rangle$.} Otherwise there are 6
elements of $\Z_7^2$, such that no two generate the same subgroup, and
the sum of two different of them is contained in two fixed cyclic
subgroups, which easily gives a contradiction.

{\it $(1,0)$, $(0,1)$ and $(2,2)$ cannot all occur with multiplicity
  $2$.} Suppose otherwise. Then the only further elements which can
occur with multiplicity 2 are $(1,6)$, $(2,4)$, $(4,2)$, $(4,6)$,
$(6,1)$, and $(6,4)$. Moreover, two elements which are exchanged by
the map $(x,y)\mapsto (y,x)$ cannot both occur in $A$, hence we may
assume that $(6,1)$ occurs twice in $A$, while $(1,6)$ does not. Then
$(2,4)$ and $(4,6)$ occur twice in $A$, and we get the zero-sum
$2\cdot(6,1)+(1,6)+(1,0)$.

{\it $(1,0)$, $(0,1)$ and $(1,1)$ cannot all occur with multiplicity
  $2$.} Using the previous result one finds that all further elements
of multiplicity 2 have one coordinate equal to 1. By symmetry we may
assume that there are two further elements of the form $(1, t)$. If
there is an element of the form $(x,y)$, $2\leq x\leq 5$, this
immediately gives a zero-sum of length $8-x$, hence all elements in
$A$ are $(1,0)$, or of the form $(1,t)$, $(6,t)$. Since there are at
least 8 different elements in $A$, there are at least 6 different
elements of the form $(x,0)$, which can be written as the sum of one
element of the form $(1,t)$ and one of the form $(6,t)$. Hence we
obtain a zero-sum of length 2 or 3.

{\it $(1,0)$, $(0,1)$ and $(4,4)$ cannot all occur with multiplicity
  $2$.} There are at least 6 elements occurring with multiplicity 2,
thus there are at least two further elements outside the subgroup
$\langle(1,-1)\rangle$. But every element different from $(2,4), (3,5),
(4,2), (5,3)$ immediately gives a zero-sum, and $(2,4)$ and $(4,2)$ as
well as $(5,3)$ and $(3,5)$ cannot both occur at the same time, thus
we may assume that $(5,3)$. The only possible element in
$\langle(3,1)\rangle$ is $(1,5)$, and this element can only occur
once. Hence $(2,4)$ becomes impossible, and we conclude that $(4,2)$
occurs with multiplicity 2. But then all elements in
$\langle(1,-1)\rangle$ yield zero-sums.

We can now finish the proof. We know that there exist two elements
$x,y\in A$, both with multiplicity 2, such that $\langle x+y\rangle$
contains an element of multiplicity 2. We may set $x=(1,0)$,
$y=(0,1)$, and let $(t,t)$ be the element in $\langle x+y\rangle$. Then
$t=0, 3, 5, 6$ immediately yields z short zero-sum, while $t=1, 2, 4$
was excluded above. Hence no counterexample exists.
\end{proof}
Now suppose that $p\geq 7$ is a prime number, and $A$ is a sequence
in $\Z_p^d$ with 
$|A|=n=\frac{p^d-p}{p^2-p}(3p-7)+4$ without zero-sums of length $\leq
p$. Let $\ell$ be a one-dimensional subgroup of $\Z_p^d$, such that
$m=|\ell\cap A|$ is maximal. Now consider all 2-dimensional subgroups
containing $\ell$. Each such subgroup contains $p^2-p$ points outside
$\ell$. Each point of $A$ is either contained in $\ell$ or occurs in
$\frac{p^2-p}{p^d-p}$ of all such subgroups. Hence among all subgroups
there is one which contains $\lceil\frac{p^2-p}{p^d-p}(n-m)\rceil$
points outside $\ell$. Call this subgroup $U$. Therefore $U$ contains
at least
\[
\left\lceil\frac{p^2-p}{p^d-p}(n-m)\right\rceil+m  \geq
\left\lceil3p-7 + m - \frac{m-4}{p^{d-2}+\dots+1}\right\rceil 
\]
elements of $A$. Since $\eta(\Z_p^2)=3p-2$, this quantity is $\leq
3p-3$, which implies $m\leq 4$. Hence $m\leq 3$, which implies that
$\frac{m-4}{p^{d-2}+\dots+1}$ is negative, and we
find that $U$ contains $3p-6+m\leq 3p-4$ points, that is, $m\leq
2$. However, this implies that each of the $p+1$ one-dimensional
subgroups of $U$ contain at most 2 elements of $A$, thus $3p-6\leq|A\cap
U|\leq 2p+2$, which implies $p\leq8$, hence, by our assumption
$p=7$. In the case $p=7$ we obtain that $U\cong\Z_7^2$ contains a
sequence $A$ of 15 elements, such that no cyclic subgroup contains more
than 2 of them, and $A$ contains no zero-sum of length $\leq 7$.

We can now prove Theorem~\ref{thm:step} for $p\geq 7$. We take
$M=\frac{p^d-p}{p^2-p}(3p-7)+4$, and let $k$ be the largest integer
for which Lemma~\ref{Lem:shortsums} ensures $D_k(\Z_p^d)\leq M$.
Then the claim of Theorem~\ref{thm:step} becomes
\[
\frac{M-2p-5}{3(p-1)/2}p+p^2 \geq M
\]
for $d=3$, and
\[
\frac{M-(6p-4)p^{d-3}\big(\frac{p^2+6p-3}{4p^2}-\frac{1}{dp}\big)}{3(p-1)/2}p
+ p^{d-1} \geq M
\]
for $d\geq 4$. After some computation one reaches the inequalities
$4p^2\geq 6p+25$ and $28p^4\geq 144p^3 +p^2-33$, which are satisfied
for $p\geq 7$. Hence the proof of Theorem~\ref{thm:step} is complete.

%

\end{document}